\documentclass[twoside,a4paper,11pt]{article}
\usepackage[T1]{fontenc}
\usepackage[utf8x]{inputenc}
\usepackage{lmodern}
\usepackage{amsmath}
\usepackage{amsfonts}
\usepackage{mathrsfs}
\usepackage{amsthm}
\usepackage{amssymb}
\usepackage{latexsym}
\usepackage{enumerate}
\usepackage{graphicx}
\usepackage{subfigure}
\usepackage{epstopdf}
\usepackage{cite}
\usepackage{color}
\usepackage[titletoc]{appendix}
\usepackage[linkcolor=black,colorlinks,hypertexnames=false,plainpages=false,citecolor=black,filecolor=black]{hyperref}
\usepackage{geometry}

\usepackage{textcomp}
\usepackage[some]{background}
\usepackage[absolute,overlay]{textpos}
\usepackage[english]{babel}
\definecolor{color1}{RGB}{31,55, 61}
\definecolor{background}{RGB}{238,237,229}

\definecolor{comp}{RGB}{98,99,101}
\definecolor{phys}{RGB}{31,55,61}
\definecolor{math}{RGB}{133,63,72}
\definecolor{urban}{RGB}{51,95,60}

\definecolor{color1}{RGB}{133,63,72}

\usepackage{pgf,tikz}
\usetikzlibrary{arrows}

\usepackage{scrextend}

\newcommand\norm[1]{\left\lVert#1\right\rVert}
\newtheorem{prop}{Proposition}[section]

\newtheorem{theorem}{Theorem}[section]
\newtheorem{lemma}{Lemma}[section]

\theoremstyle{remark}
\numberwithin{equation}{section}
\DeclareUnicodeCharacter{1013}{-}
\begin{document}
\title{Coupled Stochastic Allen-Cahn equations: Existence and Uniqueness }
\author{Tran Hoa Phu}
\newcommand{\Addresses}{{
  \bigskip
  \footnotesize

  Tran Hoa Phu, \textsc{Department of Mathematics, Gran Sasso Science Institute, Viale Francesco Crispi, 7 - 67100 L'Aquila (ITALY)}\par\nopagebreak
  \textit{E-mail address}: \texttt{hoaphu.tran@gssi.it}
}}

\maketitle
\begin{abstract}
In this paper, we concern a system of stochastic PDE's. Our system consists of two components. Each component evolves according to the sotchastic Allen-Cahn equation with a symmetric double well potential and with addtional small space-time white noise. The two components interact with each other by an attractive linear force. We aim to give a global existence and uniqueness result in a space of continuous functions on $\mathbb{R} \times \mathbb{R}^+$. We apply methods proposed by Doering \cite{Doering1987}.
\end{abstract}
\section{Introduction and main result}
We investigate the following SPDE on unbounded domain $\mathbb{R} \times \mathbb{R}^+$ 
\begin{equation}\label{EX class 1}
\begin{aligned}
&\partial_t m_1 =\dfrac{1}{2} \partial_{xx}m_1 -V'(m_1) + \lambda(m_2-m_1)+ \dot{\alpha_1}(x,t), \quad x \in \mathbb{R}, t \geq 0\\
&\partial_t m_2 =\dfrac{1}{2} \partial_{xx}m_2 -V'(m_2) +  \lambda(m_1-m_2)+  \dot{\alpha_2}(x,t), \quad x \in \mathbb{R}, t \geq 0\\
\end{aligned}
\end{equation}
where $V'(m) = m^3 -m$ is the derivative with respect to $m$ of the double well polynomial $V(m) = \dfrac{1}{4}m^4 -\dfrac{1}{2}m^2$ and $\lambda$ is a given positive constant. Two terms $\dot{\alpha_1}(x,t)$ and $\dot{\alpha_2}(x,t)$ are two independent space-time white noises. We refer to Walsh\cite{Walsh2006}, Faris and Jona-Lasinio \cite{Faris1982} and Da Prato and Zabczyk \cite{Zabc} for the SPDE's theory applied to \eqref{EX class 1}. The integral version of \eqref{EX class 1} is
\begin{equation} \label{1D: original SPDE}
\begin{aligned}
&m_{1,t} = H_t m_{1,0 }- \int_0^t H_{t-s}\Big[V'(m_{1,s}) + \lambda\Big(m_{1,s}-m_{2,s}\Big)\Big]ds + Z_{1,t}\\
&m_{2,t} = H_t m_{2,0}- \int_0^t H_{t-s}\Big[V'(m_{2,s}) + \lambda\Big(m_{2,s}-m_{1,s}\Big)\Big]ds + Z_{2,t}
\end{aligned}
\end{equation}
where 
\begin{equation}
\begin{aligned}
&H_t \; \text{is the kernel of the heat operator} \;  \partial_t - \dfrac{1}{2}\partial_{xx}, \\
 &Z_{i,t} = \int_0^t H_{t-s}\dot{\alpha}_{i,s} ds \quad \forall i=1,2\\
\end{aligned}
\end{equation}
For every $\alpha > 0$, let $\bigg( C^{\alpha}(\mathbb{R}),\norm{.}_{C^{\alpha}(\mathbb{R})} \bigg)$ and  $\bigg( C^\alpha(\mathbb{R}\times \mathbb{R}^+), \norm{.}_{C^\alpha(\mathbb{R}\times \mathbb{R}^+)} \bigg)$  be Banach spaces of continuous functions with norm
\begin{equation}
\begin{aligned}
\norm{f}_{C^{\alpha}(\mathbb{R})} &= \sup_{x \in \mathbb{R}}e^{-\alpha |x|}f(x)\\
\norm{f}_{C^{\alpha}(\mathbb{R} \times \mathbb{R}^+)} &=  \sup_{t \geq 0 } \sup_{x \in \mathbb{R}} \exp{ \bigg(-\dfrac{\alpha^2}{2} t -\alpha |x| \bigg)}|f(x,t)| 
\end{aligned}  
  \end{equation}

Let $\mathcal{F}_t$ be the filtration given by $\mathcal{F}_t :=  \sigma \lbrace Z_{i,t}(x,s); x \in \mathbb{R}, s \leq t, i=1,2 \rbrace$. We now state the main theorem
\begin{theorem}\label{main theorem model 1}
If $(m_{1,0},m_{2,0}) \in (C^{\alpha}(\mathbb{R}))^2$ for any $\alpha > 0$ then 
\begin{equation}
\begin{aligned}
&m_{1,t} = H_t m_{1,0} - \int_0^t H_{t-s}\Big[V'(m_{1,s}) + \lambda\Big(m_{1,s}-m_{2,s}\Big)\Big]ds + Z_{1,t}\\
&m_{2,t} = H_t m_{2,0} - \int_0^t H_{t-s}\Big[V'(m_{2,s}) + \lambda\Big(m_{2,s}-m_{1,s}\Big)\Big]ds + Z_{2,t}
\end{aligned}
\end{equation}
has a unique $\mathcal{F}_t$-adapted process $(m_{1,t},m_{2,t}) \in \bigg( C^{\alpha}(\mathbb{R} \times \mathbb{R}^+) \bigg)^2$ for any $\alpha > 0$.
\end{theorem}
\section{Finite Volume Equations}
Let  $F_{i,t} = H_t m_{i,0} + Z_{i,t} \; \forall i=1,2$. Since $F_{1,t}$ and $F_{2,t}$  are not bounded on unbounded domain, so it is convenient to consider the following SPDE
\begin{equation} \label{1D: FV approx}
\begin{aligned}
&m_{1,\Lambda, t} =  - \int_0^t  H_{t-s}\Lambda\Big[V'(m_{1,\Lambda,s}) + \lambda\Big(m_{1,\Lambda,s}-m_{2,\Lambda,s}\Big)\Big]ds + \Lambda F_{1,t}\\
&m_{2,\Lambda,t} =  - \int_0^t H_{t-s}\Lambda\Big[V'(m_{2,\Lambda,s}) + \lambda\Big(m_{2,\Lambda,s}-m_{1,\Lambda,s}\Big)\Big]ds + \Lambda F_{2,t}
\end{aligned}
\end{equation}
where $0 \leq \Lambda(x) \leq 1$ is a continuous function of compact support on $\mathbb{R}$.

Firstly, it is essential to establish a local existence result  by using Picard iterations as in \cite{Doering1987}.
\begin{prop}
There exists $T_0 > 0$ so that there is a continuous bounded solution to \eqref{1D: FV approx}
on $\mathbb{R} \times [0,T_0]$
\end{prop}
\begin{proof}
We define the map $P$
\begin{equation}
\begin{aligned}
(m_{1}, m_{2}) \rightarrow \bigg(&-H\Lambda  \bigg[ V'(m_{1}) + \lambda(m_{1} - m_{2}) \bigg] + \Lambda F_1,\\
&-H\Lambda  \bigg[ V'(m_{2}) + \lambda(m_{2} - m_{1}) \bigg]  + \Lambda F_2 \bigg)
\end{aligned}
\end{equation}
We shall prove the map $P$ takes the closed set 
\begin{equation}
\begin{aligned}
B=\bigg\lbrace (m_{1}, m_{2} ) \in &C(\mathbb{R}\times [0,T])^2: \\
&\norm{m_i}_\infty \leq 2\bigg( \norm{\Lambda F_1}_\infty  + \norm{\Lambda F_2}_\infty \bigg) \quad \forall i=1,2\bigg\rbrace
\end{aligned}
\end{equation}
 into itself for $T$ sufficiently small. In fact, let $(m_{1},m_{2}) \in B$ and setting $K := \norm{\Lambda F_1}_\infty + \norm{\Lambda F_2}_\infty$, we estimate 
\begin{equation}
\begin{aligned}
 &\norm{-H\Lambda[V'(m_{1}) + \lambda(m_{1} - m_{2})] + \Lambda F_1}_\infty \\
 & \leq \norm{V'(m_{1}) + \lambda(m_{1} - m_{2})}_\infty\times \int_0^{T}\int_{\mathbb{R}}H_{t-s}(x,y)dyds + \norm{\Lambda F_1}_\infty \\
 &\leq T \bigg( \norm{m_{1}}_\infty^3 + \norm{m_{1}}_\infty + \lambda (\norm{m_{1}}_\infty + \norm{m_{2}}_\infty)\bigg) + \norm{\Lambda F_1}_\infty \\
 &\leq 8T K(K^2 + \dfrac{1}{8} + \dfrac{\lambda}{4}) + \norm{\Lambda F_1}_\infty \quad (\text{note that } \norm{m_{1}}_\infty,\norm{m_{2}}_\infty \leq 2K)
\end{aligned}
\end{equation}  
Choosing $T>0$ such that 
\begin{equation}\label{2d fixed point 1}
8T(K^2 + \dfrac{1}{8} + \dfrac{\lambda}{4}) \leq 1 \Leftrightarrow T \leq \dfrac{1}{8(K^2 + 1/8 + \lambda/4)}
\end{equation}
, we derive
\begin{equation}
 \norm{-H\Lambda[V'(m_{1}) + \lambda(m_{1} - m_{2})] + \Lambda F_1}_\infty \leq K+\norm{\Lambda F_1}_\infty \leq 2K
\end{equation}
Since $V'$ is a polynomial, so it is Lipschitz continuous on bounded set in $C(\mathbb{R}\times [0,T])$. We claim that $P$ is a contraction mapping from $B \rightarrow B$ for $T_0$ sufficiently small. Taking $(m_1',m_2'),(m_1,m_2)\in B$ then
\begin{equation}
\begin{aligned}
P(m_1',m_2') &- P(m_1,m_2)\\
&= \bigg[ -H\Lambda \bigg ( V'(m_1')-V(m_1) + \lambda (m_1'-m_1 - m_2 + m_2')  \bigg),\\
&-H\Lambda \bigg ( V'(m_2')-V(m_2) + \lambda (m_2'-m_2 - m_1 + m_1')  \bigg)  \bigg]
\end{aligned}
\end{equation}  
Since
\begin{equation}
V'(a)-V'(b) = (a-b)(a^2 + ab +b^2 -1) \quad \forall a,b \in \mathbb{R}
\end{equation}
Hence, in the ball $B$
\begin{equation}
\begin{aligned}
&\norm{-H\Lambda  \bigg( V'(m_1')-V(m_1) + \lambda (m_1'-m_1 - m_2 + m_2')  \bigg)}_\infty\\
&\leq \bigg( V'(m_1')-V(m_1) + \lambda (m_1'-m_1 - m_2 + m_2')  \bigg)_\infty \int_0^{T} \int_\mathbb{R} H_{t-s}(x,y)dyds \\
&\leq T (12K^2 + 1)\norm{m_1'-m_1}_\infty + \lambda T\bigg( \norm{m_1'-m_1}_\infty + \norm{m_2'-m_2}_\infty \bigg)\\
&\leq T(12K^2 + 1+ 2\lambda) \norm{(m_1'-m_1,m_2'-m_2)}_\infty
\end{aligned}
\end{equation}
Thus, we have
\begin{equation}
\begin{aligned}
&\norm{P(m_1',m_2')-P(m_1,m_2)}_\infty\\
&\leq T(12K^2 + 1+ 2\lambda) \norm{(m_1'-m_1,m_2'-m_2)}_\infty
\end{aligned}
\end{equation}
Choosing $T$ such that 
\begin{equation}\label{2d fixed point 2}
T(12K^2 +1 + 2\lambda ) < 1 \Leftrightarrow T < \dfrac{1}{12K^2 +1 + 2\lambda}
\end{equation}
$P$ is a contraction mapping from $B \rightarrow B$. From \eqref{2d fixed point 1} and \eqref{2d fixed point 2}, $T_0$ should satisfies 
\begin{equation}
T_0 = \min \bigg( \dfrac{1}{8(K^2 + 1/8 + \lambda/4)}, \dfrac{1}{2(12K^2 +1 + 2\lambda)} \bigg)
\end{equation}
then $P$ has a unique fixed point by Contraction principle (Thereom \ref{contraction theorem}).

\end{proof}
To show global solution, it is sufficient to prove that
\begin{equation}
\sup_{t \in [0,T^*)} \sup_{x \in \mathbb{R}} |m_{i,\Lambda,t}(x)| < \infty \quad \forall i=1,2
\end{equation}
where $(m_{1,\Lambda,t},m_{2,\Lambda,t})$ is a continuous solution on $[0,T^*)$.
We recall the following useful lemma proved in Lemma 1,\cite{Doering1987}.
\begin{lemma}\label{1D: basic lemma}
If $f(x,t)$ and $(\partial_t - \dfrac{1}{2}\Delta)f$ are continuous on $\mathbb{R}\times [0,T]$ and $f(x,0)=0$, $|f|^{2n+2}$ and $|\partial_x f|^{2n+2} \in L^1(\mathbb{R} \times [0,T],dxdt)$, then
\begin{equation}
\int_0^T \int_{\mathbb{R}} f^{2n+1}(x,t)(\partial_t - \dfrac{1}{2}\Delta)f(x,t) dxdt \geq 0
\end{equation}
\end{lemma}
\begin{lemma}
If $(m_{1,\Lambda},m_{2,\Lambda})$ satisfies \eqref{1D: FV approx} on $\mathbb{R}\times [0,T]$. Then $m_{i,\Lambda}$ and $\partial_x m_{i,\Lambda}$ vanish exponentially as $|x| \rightarrow \infty$ for all $t \leq T$.
\begin{proof}
Assume $x$ is not in supp($\Lambda$) and denote $V^*(m_{1,\Lambda}) = - V'(m_{1,\Lambda}) + \lambda(m_{2,\Lambda}-m_{1,\Lambda})$. Then
\begin{equation}
m_{1,\Lambda}(x,t)  = \int_0^t \int_{\mathbb{R}} \Lambda (y) H_{t-s}(x,y)V^*(m_{1,\Lambda})(y,s)dyds
\end{equation}
This implies
\begin{equation}\label{2d vanish exponentially 1}
|m_{1,\Lambda}(x,t)| \leq \sup_{s\leq T,y\in \mathbb{R}} |\Lambda(y) V^*(m_{1,\Lambda}(y,s))| \int_0^t \int_{y \in \text{supp}\Lambda} H_{t-s}(x,y)dyds
\end{equation}
and define $d(x,\text{supp}(\Lambda)):= \inf \lbrace |x-y|:y\in \text{supp}(\Lambda) \rbrace$, then
\begin{equation}\label{2d vanish exponentially 2}
\begin{aligned}
\int_0^t \int_{y \in \text{supp}\Lambda} H_{t-s}(x,y)dyds &\leq  \int_0^t \int_{y \in \text{supp}\Lambda}\dfrac{e^{-\dfrac{d^2(x,\text{supp}(\Lambda))}{2(t-s)}}}{\sqrt{2 \pi (t-s)}}dyds \\
&\leq \norm{\text{supp}(\Lambda)} \int_0^t \dfrac{e^{-\dfrac{d^2(x,\text{supp}(\Lambda))}{2(t-s)}}}{\sqrt{2 \pi (t-s)}}ds
\end{aligned} 
\end{equation}
here $\norm{\text{supp}(\Lambda)}$ is the volume of $\text{supp}(\Lambda)$.  By changing variables, $r:= \dfrac{1}{2(t-s)}
 \Rightarrow dr= 2r^2ds$. Thus $\forall t \leq T$, we get
 \begin{equation}\label{2d vanish exponentially 3}
 \begin{aligned}
 \int_0^t \dfrac{e^{-\dfrac{d^2(x,\text{supp}(\Lambda))}{2(t-s)}}}{\sqrt{2 \pi (t-s)}}ds &\leq \int_{1/2T}^{\infty} \dfrac{e^{-rd^2(x,\text{supp}(\Lambda))}}{2r^{3/2}}dr \\
 &\leq C T^{3/2}\dfrac{e^{-d^2(x,\text{supp}(\Lambda))/2T}}{d^2(x,\text{supp}(\Lambda))}
 \end{aligned}
 \end{equation}
From \eqref{2d vanish exponentially 1},\eqref{2d vanish exponentially 2} and \eqref{2d vanish exponentially 3}, $m_{1,\Lambda}$ vanish exponentially as $|x| \rightarrow \infty$.
To see $\partial_x m_{1,\Lambda}$ also vanish exponentially, we write
\begin{equation}
\partial_x m_{1,\Lambda}(x,t) = \int_0^t \int_{\mathbb{R}} \Lambda(y) H_{t-s}(x,y)\dfrac{-(x-y)}{t-s}V^*(m_{1,\Lambda}(y,s))dyds
\end{equation}
As a result,
\begin{equation}
\begin{aligned}
|&\partial_x m_{1,\Lambda}(x,t)|\\
&\leq \sup_{s\leq T,y\in \mathbb{R}} |\Lambda(y) V^*(m_{1,\Lambda}(y,s))| \int_0^T \int_{y \in \text{supp}(\Lambda)} H_{t-s}(x,y) \dfrac{|x-y|}{t-s}dyds
\end{aligned}
\end{equation}
Moreover,
\begin{equation}
\begin{aligned}
&\int_0^t \int_{y \in \text{supp}(\Lambda)} H_{t-s}(x,y) \dfrac{|x-y|}{t-s}dyds\\
&= \int_0^t \int_{y \in \text{supp}(\Lambda)}\dfrac{e^{-\dfrac{(x-y)^2}{2(t-s)}}}{\sqrt{2 \pi (t-s)}}.\dfrac{|x-y|}{t-s}dyds\\
&\leq \dfrac{1}{2}\int_0^t \int_{y \in \text{supp}(\Lambda)} \dfrac{e^{-\dfrac{(x-y)^2}{2(t-s)}}}{\sqrt{2 \pi (t-s)}}dyds + \dfrac{1}{2}\int_0^t \int_{y \in \text{supp}(\Lambda)} \dfrac{e^{-\dfrac{(x-y)^2}{2(t-s)}}}{\sqrt{2 \pi (t-s)}}\dfrac{(x-y)^2}{(t-s)^2}dyds\\
&=: D_1 + D_2
\end{aligned}
\end{equation}
$D_1$ could be estimated as above, and we get
\begin{equation}
D_1 \leq C_1 \norm{\text{supp}(\Lambda)} T^{3/2}\dfrac{e^{-d^2(x,\text{supp}(\Lambda))/2T}}{d^2(x,\text{supp}(\Lambda))}
\end{equation}
$D_2$ is represented as below
\begin{equation}
D_2 = \dfrac{1}{\sqrt{2\pi}} \int_0^t \int_{y \in \text{supp}(\Lambda)} \dfrac{e^{-\dfrac{(x-y)^2}{2(t-s)}}}{(t-s)^{3/2}}\dfrac{(x-y)^2}{2(t-s)}dyds
\end{equation}
and using the following inequality: $xe^{-x} \leq 2 e^{-x/2}\; \forall x \in \mathbb{R}$ we have
\begin{equation}
D_2 \leq C_2 \int_0^t \int_{y \in \text{supp}(\Lambda)} \dfrac{e^{-\dfrac{(x-y)^2}{4(t-s)}}}{(t-s)^{3/2}}dyds 
\end{equation}
and by changing variables, we derive
\begin{equation}
D_2 \leq C_3 \norm{\text{supp}(\Lambda)}T^{1/2}\dfrac{e^{-d^2(x,\text{supp}(\Lambda))/4T}}{d^2(x,\text{supp}(\Lambda))}
\end{equation}
Therefore, $\partial_x m_{1,\Lambda}$ also vanishes exponentially as $|x| \rightarrow \infty$. We also have analogous properties for $m_{2,\Lambda}$.
\end{proof}
\end{lemma}

\begin{lemma} \label{bounded in Lp, finite}
 Let $(m_{1,\Lambda},m_{2,\Lambda})$  be a continuous solution to \eqref{1D: FV approx}
on $\mathbb{R}\times [0,T^*)$, then for each $\alpha >0$ and $p \geq 1$, $m_{1,\Lambda}, \; m_{2,\Lambda} \in L^p(\mathbb{R}\times [0,T],\Lambda (x)dxdt)$, i.e, t, here exists a constant $C^*=C^*(\alpha,p,T^*)$  such that
\begin{equation}\label{1D: uniform bounded T star}
\norm{m_{i,\Lambda}}_{p }^{T^*,\Lambda} \leq C \quad \forall i=1,2
\end{equation}
here $\norm{f}_{p}^{T,\Lambda} :=  \bigg( \int_0^T \int_\mathbb{R} |f|^p(x,t)\Lambda(x) dxdt \bigg)^{1/p}$
\end{lemma}
\begin{proof}By the Lemma \ref{1D: basic lemma} and let $T<T^*$, we obtain
\begin{equation} 
\begin{aligned}
0 &\leq \int_0^T \int_\mathbb{R}  (m_{1,\Lambda} - \Lambda F_1)^{2n+1}(\partial_t - \dfrac{1}{2}\partial_{xx})(m_{1,\Lambda} - \Lambda F_1) dxdt \\
&= -\int_0^T \int_{\mathbb{R}} (m_{1,\Lambda} -\Lambda F_1)^{2n+1}\Lambda \bigg(  m_{1,\Lambda}^3 - m_{1,\Lambda} + \lambda (m_{1,\Lambda} - m_{2,\Lambda}) \bigg) dxdt
\end{aligned}
\end{equation}
Then, by expanding $(m_{1,\Lambda}-\Lambda F_1)^{2n+1}$,
\begin{equation}
\begin{aligned}
0 \leq -  \int_0^T  \int_{\mathbb{R}} \Lambda(x)&  \Big[ \sum_{m=0}^{2n+1} \binom{2n+1}{m}m_{1,\Lambda}^{2n+1-m}(-\Lambda F_1)^m \Big] \\
&\times \Big[  m_{1,\Lambda}^3 -  m_{1,\Lambda}  + \lambda ( m_{1,\Lambda} - m_{2,\Lambda}) \Big] dxdt
\end{aligned}
\end{equation}
or 
\begin{equation}
\begin{aligned}
 &\int_0^T \int_{\mathbb{R}} m_{1,\Lambda}^{2n+4}\Lambda (x) dxdt \\
 & \leq \int_0^T \int_{\mathbb{R}} \Bigg[ \sum_{m=1}^{2n+1} \binom{2n+1}{m} |m_{1,\Lambda}|^{2n+4-m} |\Lambda F_1|^{m} + |1-\lambda|m_{1,\Lambda}^{2n+2}\\
&+ \lambda |m_{1,\Lambda}|^{2n+1}|m_{2,\Lambda}| \\
&+ \sum_{m=1}^{2n+1} + |1-\lambda|\binom{2n+1}{m}|m_{1,\Lambda}|^{2n+2-m} |\Lambda F_1|^m \\
&+ \lambda \sum_{m=1}^{2n+1} \binom{2n+1}{m}|m_{1,\Lambda}|^{2n+1-m} |m_{2,\Lambda}| |\Lambda F_1|^m \Bigg]\Lambda(x)dxdt
\end{aligned}
\end{equation}
Using Young's inequality,
\begin{equation}
\begin{aligned}
&\dfrac{1}{4}\bigg( \norm{m_{1,\Lambda}}_{2n+4}^{T,\Lambda}\bigg)^{2n+4} \\
&\leq \sum_{m=1}^{2n+1} \binom{2n+1}{m} \bigg(\norm{\Lambda F_1}_{2n+4}^{T,\Lambda}\bigg)^m \bigg(\norm{m_{1,\Lambda}}_{2n+4}^{T,\Lambda}\bigg)^{2n+4-m}\\
&+  |1-\lambda|\bigg(\norm{1}_{2n+4}^{T,\Lambda} \bigg)^2 \bigg( \norm{m_{1,\Lambda}}_{2n+4}^{T,\Lambda} \bigg)^{2n+2}\\
 &+ \lambda  \bigg(\norm{1}_{2n+4}^{T,\Lambda} \bigg) \bigg( \norm{m_{1,\Lambda}}_{2n+4}^{T,\Lambda} \bigg)^{2n+1} \bigg( \norm{m_{2,\Lambda}}_{2n+4}^{T,\Lambda} \bigg) \\
&+ |1-\lambda|\sum_{m=1}^{2n+1}\binom{2n+1}{m} \bigg( \norm{m_{1,\Lambda}}_{2n+4}^{T,\Lambda}\bigg)^{2n+2-m} \bigg( \norm{\Lambda F_1}_{p(m,n)}^{T,\Lambda}\bigg)^m\\
&+ \lambda \sum_{m=1}^{2n+1} \binom{2n+1}{m} \bigg( \norm{m_{1,\Lambda}}_{2n+4}^{T,\Lambda} \bigg)^{2n+1-m} \bigg( \norm{m_{2,\Lambda}}_{2n+4}^{T,\Lambda} \bigg) \bigg( \norm{\Lambda F_1}_{p(m,n)}^{T,\Lambda}\bigg)^m
\end{aligned}
\end{equation}
here $p(m,n):= \dfrac{(2n+4)m}{2+m}$. Thus,
\begin{equation}
 \bigg( \norm{m_{1,\Lambda}}_{2n+4}^{T,\Lambda} \bigg)^{2n+4} \leq P_1 \bigg(\norm{m_{1,\Lambda}}_{2n+4}^{T,\Lambda},\norm{m_{2,\Lambda}}_{2n+4}^{T,\Lambda} \bigg)
\end{equation}
In a similar way to $m_{2,\Lambda}$, we get
\begin{equation}
\bigg( \norm{m_{2,\Lambda}}_{2n+4}^{T,\Lambda} \bigg)^{2n+4} \leq P_2 \bigg(\norm{m_{1,\Lambda}}_{2n+4}^{T,\Lambda},\norm{m_{2,\Lambda}}_{2n+4}^{T,\Lambda} \bigg)
\end{equation}
where $P_1(x,y)$ and $P_2(x,y)$ is a polynomial of $x,y$ with total degree $2n+3$. 
By letting $T \rightarrow T^*$ and note that coefficients of the polynomial $P_1,P_2$ are all bounded as $T \rightarrow T^*$. Thus there exists $C$  such that
\begin{equation}
\norm{m_{i,\Lambda}}_{2n+4}^{T^*,\Lambda} \leq C \quad \forall i = 1,2
\end{equation}
\end{proof}
\begin{lemma}
The kernel $H_t(x,y) \in L^{p}(\mathbb{R}\times [0,T])$ for any $T<\infty$ and $p < 3$.
\begin{proof}
We have
\begin{equation}
\begin{aligned}
\int_0^T \int_{\mathbb{R}} H_t^p(x,y)dyds &= \int_0^T \int_{\mathbb{R}} \dfrac{e^{\dfrac{-p(x-y)^2}{2t}}}{\sqrt{2 \pi} t^{p/2}}dydt 
\end{aligned}
\end{equation}
and by changing variables, set $z:= \dfrac{\sqrt{p}(y-x)}{\sqrt{2t}}\Rightarrow dz = \dfrac{\sqrt{p}}{\sqrt{2t}}dy $. Thus
\begin{equation}
\begin{aligned}
\int_0^T \int_{\mathbb{R}} \dfrac{e^{\dfrac{-p(x-y)^2}{2t}}}{\sqrt{2 \pi} t^{p/2}}dydt &= \int_0^T \int_{\mathbb{R}} \dfrac{e^{-z^2}}{t^{p/2-1/2}}dzdt \\
&= \int_0^T \dfrac{1}{t^{p/2-1/2}}\int_{\mathbb{R}}e^{-z^2}dzdt \\
&= C \int_0^T \dfrac{1}{t^{p/2-1/2}}dt= C't^{3/2-p/2}\bigg|_{t=0}^{t=T} < \infty \Leftrightarrow p < 3.
\end{aligned}
\end{equation}
\end{proof}
\end{lemma}
\begin{lemma}\label{2D finite}
 Let $(m_{1,\Lambda},m_{2,\Lambda})$ be a continuous solution to \eqref{1D: FV approx}
on $\mathbb{R}\times [0,T^*)$, then it is bounded.
\end{lemma}
\begin{proof}  Using Holder's inequality, we get
\begin{equation}
\begin{aligned}
\norm{\chi_{[0,T^*]}m_{1,\Lambda}}_{L^{\infty}} \leq \norm{\chi_{[0,T^*]} H}_{2} &\norm{\chi_{[0,T^*]} \Lambda \bigg( V'(m_{1,\Lambda}) + \lambda (m_{1,\Lambda} - m_{2,\Lambda}) \bigg)}_{2}  \\
&+ \norm{\chi_{[0,T^*]} \Lambda F_1}_{\infty}
\end{aligned}
\end{equation}
Moreover, by $\Lambda \leq 1$, we get
\begin{equation}
\begin{aligned}
& \norm{\chi_{[0,T^*]} \Lambda \bigg(  
V'(m_{1,\Lambda}) +  \lambda (m_{1,\Lambda} - m_{2,\Lambda})
 \bigg) }_{2}  \\
 &\leq \norm{V'(m_{1,\Lambda}) + \lambda (m_{1,\Lambda} - m_{2,\Lambda})}_{2}^{T^*,\Lambda} < \infty \quad (\text{by Lemma \;} \ref{bounded in Lp, finite})
\end{aligned}
\end{equation}
and implies the result.
\end{proof}
\begin{lemma}
There is a continuous solution to  \eqref{1D: FV approx}
on $\mathbb{R} \times [0,\infty).$
\end{lemma}
\begin{proof}
This follows from Lemma \ref{2D finite}.
\end{proof}
\section{Infinite Volume Equations}
This section aims to show global existence and uniqueness  \eqref{1D: original SPDE} by letting $\Lambda(x) \rightarrow 1$ in \eqref{1D: FV approx}.
We now introduce several measures on unbounded domain $\mathbb{R} \times \mathbb{R}^+$
\begin{equation}
\begin{aligned}
\mu (dt,dx) &= e^{-\alpha^2 t/2  - \alpha |x|} dxdt \\
\mu_T (dt,dx)  &= \chi_{[0,T]}\mu (dt,dx) \\
\mu_\Lambda (dt,dx) &= \Lambda (x) \mu (dt,dx) \\
\mu_{T,\Lambda}(dt,dx) &= \chi_{[0,T]}\Lambda \mu (dt,dx)  
\end{aligned}
\end{equation}
Similar to Lemma \ref{1D: basic lemma}, we have analogous version for measure $\mu_{T,\Lambda}$ as shown in Lemma 7, \cite{Doering1987}.
\begin{lemma}\label{1D: basic lemma 2}
If $f(x,t)$ and $(\partial_t - \dfrac{1}{2}\Delta)f$ are continuous on $\mathbb{R}\times [0,T]$ and $f(x,0)=0$, $|f|^{2n+2}$ and $|\nabla f|^{2n+2} \in L^1(\mathbb{R} \times [0,T],d\mu )$, then
\begin{equation}\label{B.}
\int_{\mathbb{R^+}}\int_{\mathbb{R}} f^{2n+1}(x,t)(\partial_t - \dfrac{1}{2}\Delta)f(x,t) d \mu_{T,\Lambda}(x,t) \geq 0
\end{equation}
\end{lemma}
Let us recall that the operator $f \rightarrow Hf$ is a bounded map from $L^p(\mu) \rightarrow L^p(\mu)$ for any $p>1$. This fact was proved in Lemma 9,\cite{Doering1987}.  
\begin{lemma}\label{H Lp to Lp}
The operator $H$ with heat kernel $H_{t}(x,y)$ is a bounded map from $L^p(d\mu) \rightarrow L^p(\mu)$ for any $p>1$.
\begin{proof}
Let $g \in L^p(\mu)$ and set $f:= Hg$. We prove that $f$ is also in $L^p(\mu)$ and $\norm{H} < \infty$. We have
\begin{equation}
\begin{aligned}
f(x,t) &= \int_0^\infty \int_\mathbb{R} \chi_{[0,\infty]}(t-s)H_{t-s}(x,y) \chi_{[0,\infty]}(s)g(y,s)dyds
\end{aligned}
\end{equation}
Multiplying both sides $e^{-\alpha|x|/p - \alpha^2t/2p}$ and note that $e^{-\alpha|x|} \leq e^{\alpha |x-y|} e^{-|y|}$, we derive
\begin{equation}
\begin{aligned}
&e^{-\alpha|x|/p - \alpha^2p.t/2}|f(x,t)|\\
&\leq \int_0^\infty \int_\mathbb{R}\bigg( \chi_{[0,\infty]}(t-s)H_{t-s}(x,y) e^{\alpha|x-y|/p - \alpha^2(t-s)/2p}\\
&\times \chi_{[0,\infty]}(s)g(y,s)e^{-\alpha|y|/p - \alpha^2.s/2p}\bigg)dyds
\end{aligned}
\end{equation}
The RHS is a convolution of $g$ and $H_t$ and it is suggested to apply Young's inequality
\begin{equation}
\norm{f}_{L^p(\mu)} \leq \norm{g}_{L^p(\mu)}\int_0^\infty \int_\mathbb{R} H_t(x,y)e^{\alpha|x-y|/p - \alpha^2 t/2p}dydt
\end{equation}
Moreover,
\begin{equation}
\begin{aligned}
\norm{H}_{L^p(\mu) \rightarrow L^p(\mu)} &\leq \int_0^\infty \int_\mathbb{R} H_t(x,y)e^{\alpha|x-y|/p - \alpha^2 t/2p}dydt \\
&\leq \int_0^\infty e^{-\alpha^2t/2p}\int_\mathbb{R} \dfrac{e^{-(x-y)^2/2t + \alpha|x-y|/p}}{\sqrt{2 \pi t}}dy dt\\
&\leq \int_0^\infty e^{-\alpha^2t/2p(1-1/p)}dt \int_\mathbb{R} e^{-(z-\dfrac{\sqrt{t}\alpha}{\sqrt{2}p})^2}dz < \infty \Leftrightarrow p > 1.
\end{aligned}
\end{equation}
\end{proof}
\end{lemma}
The following important property was given in Lemma 12, \cite{Doering1987}
\begin{lemma}\label{H Lp to C}
The operator $H$ with heat kernel $H_{t}(x,y)$ is a bounded map from $L^p(d\mu) \rightarrow C^\alpha(\mathbb{R^+}\times \mathbb{R})$ for any $p>3/2$.
\begin{proof}
Let $g \in L^p(\mu)$ and set $f:= Hg$. We prove that $f$ is also in $L^p(\mu)$ and $\norm{H} < \infty$. We have
\begin{equation}
\begin{aligned}
f(x,t) &= \int_0^\infty \int_\mathbb{R} \chi_{[0,\infty]}(t-s)H_{t-s}(x,y) \chi_{[0,\infty]}(s)g(y,s)dyds
\end{aligned}
\end{equation}
Multiplying both sides $e^{-\alpha|x| - \alpha^2t/2}$ and note that $e^{-\alpha|x|} \leq e^{\alpha |x-y|} e^{-|y|}$, we derive
\begin{equation}
\begin{aligned}
&e^{-\alpha|x| - \alpha^2.t/2}|f(x,t)|\\
&\leq \int_0^\infty \int_\mathbb{R}\bigg( \chi_{[0,\infty]}(t-s)H_{t-s}(x,y) e^{\alpha|x-y| - \alpha^2(t-s)/2}\\
&\times \chi_{[0,\infty]}(s)g(y,s)e^{-\alpha|y| - \alpha^2.s/2}\bigg)dyds
\end{aligned}
\end{equation}
The RHS is a convolution of $g$ and $H_t$ and it is suggested to apply Young's inequality
\begin{equation}
\norm{f}_{C^\alpha(\mathbb{R}\times \mathbb{R^+})} \leq \norm{g}_{L^p(\mu)}\int_0^\infty \int_\mathbb{R} H_t^q(x,y)e^{q \alpha |x-y| - q\alpha^2 t/2}dydt
\end{equation}
Moreover,
\begin{equation}
\begin{aligned}
\norm{H}_{L^p(\mu) \rightarrow C^\alpha(\mathbb{R}\times \mathbb{R^+})} &\leq \int_0^\infty \int_\mathbb{R} H_t^q(x,y)e^{q\alpha|x-y| - q\alpha^2 t/2}dydt \\
&\leq \int_0^\infty \int_\mathbb{R} \dfrac{e^{-q(x-y)^2/2t + q\alpha|x-y| -q\alpha^2 t/2}}{\sqrt{2 \pi t}^q}dy dt\\
&\leq C \int_0^\infty t^{(1-q)/2}dt \int_\mathbb{R} e^{-\bigg( \dfrac{\sqrt{q}}{\sqrt{2t}}(y-x)-\dfrac{\alpha \sqrt{qt} }{\sqrt{2}}\bigg)^2}dy < \infty\\
&\Leftrightarrow q < 3 \Leftrightarrow p> 3/2. 
\end{aligned}
\end{equation}
\end{proof}
\end{lemma}
\begin{lemma} \label{bounded in Lp, finite 2}
 Let $(m_{1,\Lambda},m_{2,\Lambda})$  be a continuous solution to \eqref{1D: FV approx}
on $\mathbb{R}\times \mathbb{R^+}$, then for each $\alpha >0$ and $p \geq 1$, there exists a constant $C=C(\alpha,p)$ independent of $\Lambda$  such that
\begin{equation}\label{1D: uniform bounded T star}
\norm{m_{i,\Lambda}}_{L^p(\mu) } \leq C \quad \forall i=1,2
\end{equation}
\end{lemma}
\begin{proof}By the Lemma \ref{1D: basic lemma 2} and take arbitrarily  $T > 0$, we obtain
\begin{equation} 
\begin{aligned}
0 &\leq \int_{\mathbb{R}^+} \int_\mathbb{R}  (m_{1,\Lambda} - \Lambda F_1)^{2n+1}(\partial_t - \dfrac{1}{2}\partial_{xx})(m_{1,\Lambda} - \Lambda F_1) d \mu_{T,\Lambda}(x,t) \\
&= -\int_0^T \int_{\mathbb{R}} (m_{1,\Lambda} -\Lambda F_1)^{2n+1}\Lambda \bigg(  m_{1,\Lambda}^3 - m_{1,\Lambda} + \lambda (m_{1,\Lambda} - m_{2,\Lambda}) \bigg) d \mu(x,t)
\end{aligned}
\end{equation}
Then, by expanding $(m_{1,\Lambda}-\Lambda F_1)^{2n+1}$,
\begin{equation}
\begin{aligned}
0 \leq - & \int_0^T  \int_{\mathbb{R}} \Lambda(x) \Bigg[ \sum_{m=0}^{2n+1} \binom{2n+1}{m}m_{1,\Lambda}^{2n+1-m}(-\Lambda F_1)^m \Bigg] \\
&\times \Big[  m_{1,\Lambda}^3 -  m_{1,\Lambda}  + \lambda ( m_{1,\Lambda} - m_{2,\Lambda}) \Big] d\mu(x,t)
\end{aligned}
\end{equation}
Take the highest order term above to the left side to derive the estimate 
\begin{equation}
\begin{aligned}
 &\int_0^T \int_{\mathbb{R}} m_{1,\Lambda}^{2n+4}\Lambda (x) d \mu(x,t) \\
 &\leq \int_0^T \int_{\mathbb{R}} \Bigg[ \sum_{m=1}^{2n+1} \binom{2n+1}{m} |m_{1,\Lambda}|^{2n+4-m} |\Lambda F_1|^{m} \\
 &+ |1-\lambda|m_{1,\Lambda}^{2n+2} + \lambda |m_{1,\Lambda}|^{2n+1}|m_{2,\Lambda}| \\
 &+ \sum_{m=1}^{2n+1} + |1-\lambda|\binom{2n+1}{m}|m_{1,\Lambda}|^{2n+2-m} |\Lambda F_1|^m \\
 &+ \lambda \sum_{m=1}^{2n+1} \binom{2n+1}{m}|m_{1,\Lambda}|^{2n+1-m} |m_{2,\Lambda}| |\Lambda F_1|^m \Bigg]\Lambda(x)d \mu(x,t)
\end{aligned}
\end{equation}
and using Young's inequality,
\begin{equation}
\begin{aligned}
&\dfrac{1}{4}\bigg( \norm{m_{1,\Lambda}}_{L^{2n+4}(\mu_{T,\Lambda})}\bigg)^{2n+4} \\
&\leq \sum_{m=1}^{2n+1} \binom{2n+1}{m} \bigg(\norm{\Lambda F_1}_{L^{2n+4}(\mu_{T,\Lambda})}\bigg)^m \bigg(\norm{m_{1,\Lambda}}_{L^{2n+4}(\mu_{T,\Lambda})}\bigg)^{2n+4-m}\\
&+  |1-\lambda|\bigg(\norm{1}_{L^{2n+4}(\mu_{T,\Lambda})} \bigg)^2 \bigg( \norm{m_{1,\Lambda}}_{L^{2n+4}(\mu_{T,\Lambda})} \bigg)^{2n+2}\\
 &+ \lambda  \bigg(\norm{1}_{L^{2n+4}(\mu_{T,\Lambda})} \bigg) \bigg( \norm{m_{1,\Lambda}}_{L^{2n+4}(\mu_{T,\Lambda})} \bigg)^{2n+1} \bigg( \norm{m_{2,\Lambda}}_{L^{2n+4}(\mu_{T,\Lambda})} \bigg) \\
&+ |1-\lambda|\sum_{m=1}^{2n+1}\binom{2n+1}{m} \bigg( \norm{m_{1,\Lambda}}_{L^{2n+4}(\mu_{T,\Lambda})}\bigg)^{2n+2-m} \bigg( \norm{\Lambda F_1}_{L^{p(m,n)}(\mu_{T,\Lambda})}\bigg)^m \quad \\
&+ \lambda \sum_{m=1}^{2n+1} \binom{2n+1}{m} \bigg( \norm{m_{1,\Lambda}}_{L^{2n+4}(\mu_{T,\Lambda})} \bigg)^{2n+1-m}\times\\
&\times \bigg( \norm{m_{2,\Lambda}}_{L^{2n+4}(\mu_{T,\Lambda})} \bigg) \bigg( \norm{\Lambda F_1}_{L^{p(m,n)}(\mu_{T,\Lambda})}\bigg)^m
\end{aligned}
\end{equation}
here $p(m,n):= \dfrac{(2n+4)m}{2+m}$. Thus,
\begin{equation}
 \bigg( \norm{m_{1,\Lambda}}_{L^{2n+4}(\mu_{T,\Lambda})} \bigg)^{2n+4} \leq \overline{P}_1 \bigg(\norm{m_{1,\Lambda}}_{L^{2n+4}(\mu_{T,\Lambda})},\norm{m_{2,\Lambda}}_{L^{2n+4}(\mu_{T,\Lambda})} \bigg)
\end{equation}
where $P(x,y)$ is a polynomial of $x,y$ with total degree $2n+3$.

In a similar way to $m_{2,\Lambda}$, we get
\begin{equation}
\bigg( \norm{m_{2,\Lambda}}_{L^{2n+4}(\mu_{T,\Lambda})} \bigg)^{2n+4} \leq \overline{P}_2 \bigg(\norm{m_{2,\Lambda}}_{L^{2n+4}(\mu_{T,\Lambda})},\norm{m_{1,\Lambda}}_{L^{2n+4}(\mu_{T,\Lambda})} \bigg)
\end{equation}
  note that coefficients of the polynomial $\overline{P}_1,\overline{P}_2$ are all bounded as $T \rightarrow \infty$ and $\Lambda \rightarrow 1$
\begin{equation}
\norm{\Lambda F_i}_{L^p(\mu_{T,\Lambda})}  \leq \norm{F_i}_{L^p(\mu)} < \infty,\quad \norm{1}_{L^p(\mu_{T,\Lambda})} \leq \norm{1}_{L^p(\mu)} < \infty
\end{equation}
Thus there exists $C'$ independent of $T$ and $\Lambda
$ such that
\begin{equation}
\norm{m_{i,\Lambda}}_{L^{2n+4}(\mu_{T,\Lambda})} \leq C' \quad \forall i = 1,2
\end{equation}
So there is $C=C(\alpha,p) > 0$
\begin{equation}
\norm{m_{i,\Lambda}}_{L^p(\mu_{T,\Lambda})} \leq C \quad \forall i=1,2
\end{equation}
Since $\Lambda \leq 1$ and \eqref{1D: uniform bounded T star} and letting $T \rightarrow \infty$
\begin{equation} \label{1D: 2-bounded}
\norm{\Lambda m_{i,\Lambda}}_{L^p(\mu)} \leq \lim_{T \rightarrow \infty} \norm{m_{i,\Lambda}}_{L^p(\mu_{T,\Lambda})} < \infty
\end{equation}
Moreover, the heat operator $H$ is bounded from $L^p(\mu)$ to $L^p(\mu)$\; (Lemma \ref{H Lp to Lp}). Thus, from 
\begin{equation}
\norm{m_{i,\Lambda}}_{L^p(\mu)} \leq \norm{H} \norm{\Lambda \bigg(V'(m_{i,\Lambda}) + \lambda (m_{i,\Lambda} - m_{i+1,\Lambda}) \bigg)}_{L^p(\mu)} + \norm{ F_i}_{L^p(\mu)}
\end{equation}
here $\norm{H}=\norm{H}_{L^p(\mu) \rightarrow L^p(\mu)}$. This is combined with \eqref{1D: 2-bounded}, the lemma follows.
\end{proof}
\begin{lemma}\label{Cauchy convergence lemma}
 	For $\beta$ large enough and for all $i=1,2$,  $e^{-\beta t} \Lambda m_{i,\Lambda}$ and $e^{-\beta t}\Lambda m_{i,\Lambda}$ are Cauchy in $L^p(d\mu)$ as $\Lambda \rightarrow 1$.
\end{lemma}
\begin{proof}
Rewrite \eqref{1D: FV approx} into PDE
\begin{equation}
(\partial_t - \dfrac{1}{2}\partial_{xx})(m_{1,\Lambda} - \Lambda F_1)  = -\Lambda \bigg(  V'(m_{1,\Lambda}) + \lambda (m_{1,\Lambda} - m_{2,\Lambda}) \bigg)
\end{equation}
it follows that for any $\beta > 0$
\begin{equation}
\begin{aligned}
&(\partial_t - \dfrac{1}{2}\partial_{xx})\bigg( e^{-\beta t} (m_{1,\Lambda}-\Lambda F_1)\bigg)\\
&= -e^{-\beta t} \Lambda \bigg( V'(m_{1,\Lambda}) + \beta m_{1,\Lambda}  + \lambda (m_{1,\Lambda} - m_{2,\Lambda})\bigg)- \beta e^{-\beta t}(1-\Lambda)m_{1,\Lambda}\\
&+ \beta e^{-\beta t}\Lambda F_1
\end{aligned}
\end{equation}
Thus
\begin{equation}
\begin{aligned}
&(\partial_t -\dfrac{1}{2}\partial_{xx})\bigg( e^{-\beta t}(m_{1,\Lambda'} - m_{1,\Lambda})-e^{\beta t}(\Lambda'-\Lambda)F_1 \bigg)\\
&= -e^{-\beta t}\bigg \lbrace\Lambda' \bigg(V'(m_{1,\Lambda'})+\beta m_{1,\Lambda}' + \lambda(m_{1,\Lambda'}-m_{2,\Lambda'})\bigg)\\
&- \Lambda \bigg(V'(m_{1,\Lambda})+\beta m_{1,\Lambda}+ \lambda(m_{1,\Lambda}-m_{2,\Lambda}\bigg) \bigg \rbrace\\ &-\beta e^{-\beta t}\bigg( (1-\Lambda')m_{1,\Lambda'}-(1-\Lambda)m_{1,\Lambda}\bigg)\\
&+\beta e^{-\beta t} (\Lambda' - \Lambda)F_1
\end{aligned}
\end{equation}
Using the Lemma \ref{1D: basic lemma 2}, we get
\begin{equation}\label{1 lemma}
\begin{aligned}
0 \leq &\int_{\mathbb{R^+}}\int_\mathbb{R} \bigg( e^{-\beta t} (m_{1,\Lambda}- m_{1,\Lambda'})- e^{-\beta t}(\Lambda' - \Lambda)F_1\bigg)^{2n+1} \times \\
& \bigg( \partial_t - \dfrac{1}{2}\Delta \bigg)\bigg( e^{-\beta t}(m_{1,\Lambda'}-m_{1,\Lambda'}) - e^{-\beta t} (\Lambda' - \Lambda)F_1\bigg) d\mu \\ 
&=- \int_{\mathbb{R^+}}\int_\mathbb{R}  e^{-(2n+2)\beta t} \bigg((m_{1,\Lambda'} -m_{1,\Lambda}) - (\Lambda' -\Lambda)F_1 \bigg)^{2n+1} \times \\
&\bigg[ \Lambda' \bigg( V'(m_{1,\Lambda'}) + \beta m_{1,\Lambda'}- \lambda (m_{2,\Lambda'}-m_{1,\Lambda'})\bigg) \\
& - \Lambda \bigg(V'(m_{1,\Lambda}) + \beta m_{1,\Lambda}- \lambda (m_{2,\Lambda}-m_{1,\Lambda}) \bigg) \bigg] d\mu \\
&= -\int_{\mathbb{R^+}}\int_\mathbb{R} e^{-(2n+2)\beta t} \bigg( \Lambda' m_{1,\Lambda'} - \Lambda m_{1,\Lambda}\bigg)^{2n+1} \times \\
&\bigg[ 
\bigg(V'(\Lambda' m_{1,\Lambda'}) + \beta \Lambda' m_{1,\Lambda'} -\lambda (\Lambda' m_{2,\Lambda'} - \Lambda' m_{1,\Lambda'}) \bigg) \\
&- \bigg(V'(\Lambda m_{1,\Lambda}) + \beta \Lambda m_{1,\Lambda} -\lambda (\Lambda m_{2,\Lambda} - \Lambda m_{1,\Lambda}) \bigg)   
\bigg] d\mu + R_1 \\ 
&= - \int_{\mathbb{R^+}}\int_\mathbb{R}  e^{-(2n+2) \beta t} (\Lambda' m_{1,\Lambda'} - \Lambda m_{1,\Lambda})^{2n+1}(A_1 + A_2)d\mu + R_1
\end{aligned}
\end{equation}
where
\begin{equation}
\begin{aligned}
A_1 = V'(\Lambda' m_{1,\Lambda'}) + \beta \Lambda' m_{1,\Lambda'} - V'(\Lambda m_{1,\Lambda}) - \beta \Lambda m_{1,\Lambda} \\
A_2 = -\lambda \bigg( \Lambda' m_{2,\Lambda'} - \Lambda m_{2,\Lambda} - (\Lambda' m_{1,\Lambda'} - \Lambda m_{1,\Lambda})   \bigg)
\end{aligned}
\end{equation}
and  $R_1$ consists of terms vanishing as $\Lambda,\Lambda' \rightarrow 1$. Setting

For $\beta$ large enough, $V'+\beta Id$ is a monopolynomial, so there exists $k>0$ such that
\begin{equation}
(V'(x) + \beta x - V'(y)-\beta y)(x-y) \geq k(x-y)^4 \quad \forall x,y \in \mathbb{R}
\end{equation}
Setting $X_1 = \Lambda' m_{1,\Lambda'} - \Lambda m_{1,\Lambda}$ and  $X_2 =  \Lambda' m_{2,\Lambda'} - \Lambda m_{2,\Lambda}$
\begin{equation}  \label{eqn: 2D: 1 term}
k\int_{\mathbb{R^+}}\int_\mathbb{R}  e^{-(2n+2)\beta t} X_1^{2n+4}d\mu + \lambda \int_{\mathbb{R^+}}\int_\mathbb{R}  e^{-(2n+2)\beta t}X_1^{2n+1}(X_1 -X_2) d\mu \leq R_1
\end{equation}
We also get similar expression for $X_2$
\begin{equation}  \label{eqn: 2D: 2 term}
k\int_{\mathbb{R^+}}\int_\mathbb{R}  e^{-(2n+2)\beta t} X_2^{2n+4}d\mu + \lambda \int_{\mathbb{R^+}}\int_\mathbb{R}  e^{-(2n+2)\beta t}X_2^{2n+1}(X_2 -X_1)d\mu \leq R_2
\end{equation}
Taking the sum of \eqref{eqn: 2D: 1 term} and \eqref{eqn: 2D: 2 term}
\begin{equation}
\begin{aligned}
&k \int_{\mathbb{R^+}}\int_\mathbb{R}  e^{-(2n+2)\beta t}(X_1^{2n+4} + X_2^{2n+4})d\mu \\
& +  \lambda \int_{\mathbb{R^+}}\int_\mathbb{R}  e^{-(2n+2)\beta t}(X_1^{2n+1}-X_2^{2n+1})(X_1-X_2)d\mu \leq R_1 + R_2
\end{aligned}
\end{equation}
It is easy to see that the second term of LHS is non-negative and can be neglected 
\begin{equation}
k \int_{\mathbb{R^+}}\int_\mathbb{R} e^{-(2n+2)\beta t}(X_1^{2n+4} + X_2^{2n+4})d\mu  \leq R_1 + R_2
\end{equation}
Since $e^{-(2n+2)\beta t} \geq e^{-(2n+4) \beta t}$, we derive
\begin{equation}
 \norm{e^{-\beta t} (\Lambda' m_{i,\Lambda'} - \Lambda m_{i,\Lambda})}_{L^{2n+4}(\mu)}^{2n+4} \leq \int_{\mathbb{R^+}}\int_\mathbb{R} e^{-(2n+2)\beta t}X_i^{2n+4}d\mu \leq \dfrac{R_1 + R_2}{k} \rightarrow 0
\end{equation}
as $\Lambda,\Lambda' \rightarrow 1$, which implies the statement.
\end{proof}
With help of above lemma, we are ready to establish global existence and uniqueness in space $C^{\alpha}(\mathbb{R} \times \mathbb{R}^+)$.

\subsubsection{Proof of Theorem \ref{main theorem model 1}}
\begin{proof}(Existence)
By Lemma \ref{Cauchy convergence lemma}, there is a $\beta > 0$ such that 
\begin{equation}
e^{-\beta t} \Lambda m_{i,\Lambda} \rightarrow e^{-\beta t} m_i \quad  \text{in $L^p(d \mu)$}
\end{equation}
but $e^{-\beta t}\Lambda m_{i,\Lambda} = e^{-\beta t}m_{i,\Lambda} + e^{-\beta t}(1-\Lambda) m_{i,\Lambda}$ and this follows
\begin{equation}\label{equation 1}
e^{-\beta t} m_{i,\Lambda} \rightarrow e^{-\beta t}m_{i} \quad \text{in $L^p(d \mu)$} \quad \text{as} \; \Lambda \rightarrow 1 
\end{equation} 
If $\beta < \alpha^2/2$, we choose $q,r \geq 1$ satisfies $\dfrac{1}{p} = \dfrac{1}{q} + \dfrac{1}{r}$, $1<q<\dfrac{\alpha^2}{2\beta}$ and applying Young's inequality
\begin{equation}
\begin{aligned}
\norm{m_{i,\Lambda}-m_i}_{L^p(\mu)} &= \norm{e^{\beta t } \bigg( e^{-\beta t}m_{i,\Lambda} - e^{-\beta t}m_{i,\Lambda} \bigg)}_{L^p(\mu)} \\
&\leq \norm{e^{\beta t}}_{L^q(\mu)} \norm{e^{-\beta t} m_{i,\Lambda} - e^{-\beta t}m_i}_{L^r(\mu)}
\end{aligned}
\end{equation}
Since $1 < q < \dfrac{\alpha^2}{2\beta}$, this implies $\quad \norm{e^{\beta t}}_{L^q(\mu)} < \infty$ and from \eqref{equation 1}, we obtain as $\Lambda \rightarrow 1$
\begin{equation}
m_{i,\Lambda} \rightarrow m_i \quad \text{in} \; L^p(\mu) \quad \forall i=1,2
\end{equation}  
Therefore, \eqref{1D: original SPDE} has a solution $(m_1,m_2)$ in $(L^p(\mu))^2$. Moreover, the operator $H$ is a bounded map from $L^{p}$ to $C^{\alpha}(\mathbb{R}\times \mathbb{R}^+)$ (Lemma \ref{H Lp to C}), \eqref{1D: original SPDE} has also a continuous solution in $\bigg( C^\alpha(\mathbb{R}\times \mathbb{R}^+) \bigg)^2$ with $\alpha^2/2 > \beta$.

To show \eqref{1D: original SPDE} has a solution in $\bigg( C^{\alpha}(\mathbb{R}\times \mathbb{R}^+)\bigg)^2$ for  every $\alpha > 0$. Using the Lemma \ref{bounded in Lp, finite 2}, there exists $(m_{1}',m_{2}')$ and a weakly convergence sequence such that
\begin{equation}
m_{i,\Lambda_n} \rightharpoonup m_i' \quad \text{in $L^p(d \mu)$} \quad \forall i=1,2
\end{equation}
This is combined with \eqref{equation 1}, we get $m_i' = m_i$ a.e and hence $m_i \in L^p(d \mu)$ and $H$ is closed by the Lemma \ref{H Lp to C}. Furthermore, the operator $H$ is a bounded map from $L^p(\mu)$ to $C^{\alpha}(\mathbb{R}\times \mathbb{R}^+)$(Lemma \ref{H Lp to C}),the system has a continuous solution $(m_1,m_2)$ in $\bigg( C^\alpha(\mathbb{R}\times \mathbb{R}^+) \bigg)^2$.

(Uniqueness) Let $(m_1,m_2)$ and $(m_1',m_2')$ be two solutions. Multiplying by $e^{-\beta t}$ for a large enough $\beta$ and use the same computation in the Lemma \ref{Cauchy convergence lemma}, we get $e^{-\beta t } m_i = e^{-\beta t}m_i' \quad \forall i=1,2$. Hence $m_i = m_i' \quad \forall i=1,2$. Indeed, we have
\begin{equation}
(\partial_t - \dfrac{1}{2} \partial_{xx})(m_1- F_1)= -V'(m_1) + \lambda(m_2-m_1) 
\end{equation}
Then
\begin{equation}
(\partial_t - \dfrac{1}{2}\partial_{xx})\bigg( e^{-\beta t}(m_1 - F_1) \bigg) = -e^{-\beta t}\bigg( V'(m_1)+\beta m_1 + \lambda(m_1 - m_2) \bigg) +\beta e^{-\beta t}F_1
\end{equation}
As a result,
\begin{equation}
\begin{aligned}
(&\partial_t - \dfrac{1}{2}\partial_{xx})\bigg(e^{-\beta t}(m_1'-m_1') \bigg)\\
&= -e^{-\beta t}\bigg \lbrace \bigg( V'(m_1')+\beta m_1' + \lambda(m_1' - m_2') \bigg)\\
&-\bigg( V'(m_1)+\beta m_1  + \lambda(m_1 - m_2)\bigg) \bigg \rbrace 
\end{aligned}
\end{equation}
Applying the Lemma \ref{1D: basic lemma 2},
\begin{equation}
\begin{aligned}
0 &\leq \int_\mathbb{R^+}\int_\mathbb{R} \bigg( e^{-\beta t}(m_1' - m_1)\bigg)^{2n+1}(\partial_t - \dfrac{1}{2}\partial_{xx})\bigg( e^{-\beta t}(m_1'-m_1) \bigg)d\mu\\
&= -\int_\mathbb{R^+}\int_\mathbb{R} e^{-(2n+2)\beta t}(m_1'-m_1)^{2n+1}\bigg( V'(m_1')+\beta m_1' -V'(m_1)-\beta m_1\bigg)d\mu \\
& - \int_\mathbb{R^+}\int_\mathbb{R} e^{-(2n+2)\beta t}(m_1'-m_1)^{2n+1}\bigg( m_1'-m_1' - (m_2'- m_2) \bigg)d\mu =: -\overline{A}_1 - \overline{A}_2
\end{aligned}
\end{equation}
where
\begin{equation}
\begin{aligned}
\overline{A}_1 &:= \int_\mathbb{R^+}\int_\mathbb{R} e^{-(2n+2)\beta t}(m_1'-m_1)^{2n+1}\bigg( V'(m_1')+\beta m_1' -V'(m_1)-\beta m_1\bigg)d\mu \\
\overline{A}_2 &:= \int_\mathbb{R^+}\int_\mathbb{R} e^{-(2n+2)\beta t}(m_1'-m_1)^{2n+1}\bigg( m_1'-m_1' - (m_2'- m_2) \bigg)d\mu 
\end{aligned}
\end{equation}
Setting $\overline{X}_1 : = m_1'-m_1$ and $\overline{X}_2 := m_2' - m_2$ and recalling 
\begin{equation}
(V'(x)+\beta x - V'(y) -\beta y)(x-y) \geq k(x-y)^4 \quad \forall x,y \in \mathbb{R}
\end{equation}
We get
\begin{equation}
\overline{A}_1 \geq \int_\mathbb{R^+}\int_\mathbb{R} e^{-(2n+2)\beta t}(m_1' - m_1)^{2n+4}d\mu \geq \norm{e^{-\beta t}(m_1'-m_1)}_{L^{2n+4}(\mu)}^{2n+4}
\end{equation}
and
\begin{equation}
\overline{A}_2 = \int_\mathbb{R^+}\int_\mathbb{R}e^{-(2n+2)\beta t}X_1^{2n+1}(X_1-X_2)d\mu
\end{equation}
and from $\overline{A}_1 + \overline{A}_2 \leq 0$, we derive
\begin{equation}\label{2d: overline A1}
\norm{e^{-\beta t}(m_1'-m_1)}_{L^{2n+4}(\mu)}^{2n+4} +\int_\mathbb{R^+}\int_\mathbb{R}e^{-(2n+2)\beta t}X_1^{2n+1}(X_1-X_2)d\mu \leq 0 
\end{equation}
Analogously, 
\begin{equation}\label{2d: overline A2}
\norm{e^{-\beta t}(m_2'-m_2)}_{L^{2n+4}(\mu)}^{2n+4} +\int_\mathbb{R^+}\int_\mathbb{R}e^{-(2n+2)\beta t}X_2^{2n+1}(X_2-X_1)d\mu \leq 0 
\end{equation}
Taking the sum of \eqref{2d: overline A1} and \eqref{2d: overline A2}, we have
\begin{equation}
\begin{aligned}
&\norm{e^{-\beta t}(m_1'-m_1)}_{L^{2n+4}(\mu)}^{2n+4} + \norm{e^{-\beta t}(m_2'-m_2)}_{L^{2n+4}(\mu)}^{2n+4}\\
&+ \int_\mathbb{R^+}\int_\mathbb{R}e^{-(2n+2)\beta t}\bigg( X_1^{2n+1}-X_2^{2n+1} \bigg)(X_1-X_2)d\mu \leq 0
 \end{aligned}
\end{equation}
Since the third term on LHS is non-negative, it is easy to have
\begin{equation}
\norm{e^{-\beta t}(m_i'-m_i)}_{L^{2n+4}(\mu)} = 0 \quad \forall i=1,2
\end{equation}
As a result,
\begin{equation}
m_i' = m_i \quad \forall i =1,2
\end{equation}
\end{proof}
\phantomsection
\addcontentsline{toc}{chapter}{Bibliography}\bibliographystyle{plain}
\bibliography{Thesis_ref}

\begin{thebibliography}{1}

\bibitem{Doering1987}
Charles~R. Doering.
\newblock {Nonlinear parabolic stochastic differential equations with additive
  colored noise on Rd ×R+: A regulated stochastic quantization}.
\newblock {\em Communications in Mathematical Physics}, 109(4):537--561, 1987.

\bibitem{Faris1982}
W.~G. Faris and G.~Jona-Lasinio.
\newblock {Large fluctuations for a nonlinear heat equation with noise}.
\newblock {\em Journal of Physics A: Mathematical and General},
  15(10):3025--3055, 1982.

\bibitem{Zabc}
Jerzy~Zabczyk {G. Da Prato}, editor.
\newblock {\em {Stochastic Equations in Infinite Dimensions}}.

\bibitem{Walsh2006}
John~B. Walsh.
\newblock {An introduction to stochastic partial differential equations}.
\newblock {\em {\'{E}}cole d'{\'{E}}t{\'{e}} de Probabilit{\'{e}}s de Saint
  Flour XIV - 1984}, pages 265--439, 2006.

\end{thebibliography}
\Addresses
\end{document}